\documentclass[a4paper,10pt]{article}

\usepackage{bbm}
\usepackage{mathrsfs}
\usepackage{amsmath,amsfonts,amssymb,amsthm,textcomp}
\usepackage[numbers]{natbib}
\usepackage[breaklinks=true]{hyperref}
\numberwithin{equation}{section}
\allowdisplaybreaks[4]

\newcommand{\IR}{\mathbbm{R}}

\newcommand{\bigo}{\mathrm{O}}
\newcommand{\lito}{\mathrm{o}}

\newcommand{\toinf}{\to\infty}

\newcommand{\eps}{\varepsilon}
\renewcommand{\phi}{\varphi}

\newcommand{\I}{\mathrm{I}}

\newcommand{\ahalf}{{\textstyle\frac{1}{2}}}
\newcommand{\athird}{{\textstyle\frac{1}{3}}}

\newcommand{\asixth}{{\textstyle\frac{1}{6}}}
\newcommand{\anth}{{\textstyle\frac{1}{n}}}

\newcommand{\eq}{\eqref}

\newcommand{\IE}{\mathbbm{E}}
\newcommand{\IP}{\mathbbm{P}}
\newcommand{\Var}{\mathop{\mathrm{Var}}\nolimits}
\newcommand{\Cov}{\mathop{\mathrm{Cov}}}

\newcommand{\Bi}{\mathop{\mathrm{Bi}}}
\newcommand{\MVN}{\mathrm{MVN}}

\def\be#1\ee{\begin{equation*}#1\end{equation*}}
\def\ben#1\ee{\begin{equation}#1\end{equation}}

\def\bes#1\ee{\begin{equation*}\begin{split}#1\end{split}\end{equation*}}
\def\besn#1\ee{\begin{equation}\begin{split}#1\end{split}\end{equation}}

\def\bg#1\ee{\begin{gather*}#1\end{gather*}}
\def\bgn#1\ee{\begin{gather}#1\end{gather}}

\def\bm#1\ee{\begin{multline*}#1\end{multline*}}
\def\bmn#1\ee{\begin{multline}#1\end{multline}}

\def\ba#1\ee{\begin{align*}#1\end{align*}}
\def\ban#1\ee{\begin{align}#1\end{align}}

\def\klr#1{(#1)}
\def\bklr#1{\bigl(#1\bigr)}

\def\bbbklr#1{\biggl(#1\biggr)}

\def\bkle#1{\bigl[#1\bigr]}

\def\bbbkle#1{\biggl[#1\biggr]}

\def\klg#1{\{#1\}}
\def\bklg#1{\bigl\{#1\bigr\}}

\def\bbbklg#1{\biggl\{#1\biggr\}}

\def\bklgl{\bigl\{}

\def\bklgr{\bigr\}}

\def\norm#1{\Vert#1\Vert}

\def\tnsk{\kern-0.1ex}

\def\VERT{{|\tnsk{\kern-0.2ex}|\tnsk{\kern-0.2ex}|}}
\def\bVERT{{\big|\tnsk{\kern-0.2ex}\big|\tnsk{\kern-0.2ex}\big|}}
\def\bbVERT{{\Big|\tnsk{\kern-0.2ex}\Big|\tnsk{\kern-0.2ex}\Big|}}
\def\bbbVERT{{\bigg|\tnsk{\kern-0.2ex}\bigg|\tnsk{\kern-0.2ex}\bigg|}}
\def\bbbVERT{{\Bigg|\tnsk{\kern-0.2ex}\Bigg|\tnsk{\kern-0.2ex}\Bigg|}}

\def\tnorm#1{\VERT#1\VERT}

\def\abs#1{\vert#1\vert}
\def\babs#1{\bigl\vert#1\bigr\vert}
\def\bbabs#1{\Bigl\vert#1\Bigr\vert}
\def\bbbabs#1{\biggl\vert#1\biggr\vert}

\def\bmid{\bigm\vert}

\def\%#1{\mathcal{#1}}
\def\^#1{\ifmmode{\mathaccent"705E #1}\else{\accent94 #1}\fi}
\def\~#1{\ifmmode{\mathaccent"707E #1}\else{\accent"7E #1}\fi}
\edef\-#1{\noexpand\ifmmode {\noexpand\bar{#1}} \noexpand\else%
\-#1\noexpand\fi}

\def\wt#1{\widetilde{#1}}

\def\leq{\leqslant}

\def\geq{\geqslant}

\theoremstyle{plain}
\newtheorem{theorem}{Theorem}[section]

\newtheorem{lemma}[theorem]{Lemma}
\newtheorem{corollary}[theorem]{Corollary}

\theoremstyle{definition}
\newtheorem{definition}{Definition}[section]

\newtheorem{remark}[definition]{Remark}

\makeatletter 

\renewcommand\section{\@startsection {section}{1}{\z@}%
{-3.5ex \@plus -1ex \@minus -.2ex}%
{1.3ex \@plus.2ex}%
{\center\small\sc\mathversion{bold}\MakeUppercase}}

\def\subsection#1{\@startsection {subsection}{2}{0pt}%
{-3.5ex \@plus -1ex \@minus -.2ex}%
{1ex \@plus.2ex}%
{\bf\mathversion{bold}}{#1}}

\def\subsubsection#1{\@startsection{subsubsection}{3}{0pt}%
{\medskipamount}%
{-10pt}%
{\normalsize\itshape}{\kern-2.2ex. #1.}}

\makeatother

\def\cite{\citet}

\usepackage{setspace}

\begin{document}

\title{\sc\bf\large\MakeUppercase{Stein's method in high~dimensions 
with~applications}}
\author{\sc Adrian R\"ollin}
\date{\it National University of Singapore}

\maketitle

\begin{abstract} 
Let $h$ be a three times partially differentiable function on $\IR^n$, let
$X=(X_1,\dots,X_n)$ be a collection of real-valued random variables and let
$Z=(Z_1,\dots,Z_n)$ be a multivariate Gaussian vector. In this article, we
develop Stein's method to give error bounds on the difference $\IE h(X) - \IE
h(Z)$ in cases where the coordinates of $X$ are not necessarily independent,
focusing on the high dimensional case $n\toinf$. In order to express the
dependency structure we use Stein couplings, which allows for a broad range of
applications, such as classic occupancy, local dependence, Curie-Weiss model
etc. We will also give applications to the Sherrington-Kirkpatrick model and
last passage percolation on thin rectangles.\\

Soi $h$ une fonction r\'eelles sur $\IR^n$ dont les d\'eriv\'es partielles
d'ordre trois existent, soi $X=(X_1,\dots,X_n)$ un vecteur des variables
al\'eatoire r\'eelles et soi $Z=(Z_1,\dots,Z_n)$ un vecteur des  variables
al\'eatoire r\'eelles suivant la loi gaussienne. Dans cet article, on \'etablit
la m\'ethode de Stein pour obtenir une majoration de la difference $\IE
h(X) - \IE h(Z)$ au cas o\`u les coordinates de $X$ ne sont pas
n\'ecessairement independentes; nous concentrons sur le cas de grande dimension
$n\toinf$. Pour exprimer la structure de dependence, on fait usage des couplages
de Stein, ce qui permet une large gamme d'application, par exemple au mod\`eles
des urnes, dependence locale, le mod\`ele de Curie-Weiss etc. Nous discutons
aussi bien des applications au mod\`ele de Sherrington-Kirkpatrick et ultime
passage de percolation en rectangles \'etroits.

\end{abstract}

\bigskip

\hfil\vbox{\hsize=0.87\hsize\noindent\textit{Keywords:} 
Stein's method;  Gaussian interpolation; last passage
percolation on thin rectangles; Sherrington-Kirkpatrick model;  Curie-Weiss
model}\hfil\hfil


\section{Introduction}

Let $X$ and $Z$ be random vectors in $\IR^n$ and let $h:\IR^n\to\IR$ be a
function of interest. A fundamental problem in probability theory is to obtain
bounds on the quantity
\ben								\label{1}
  \abs{\IE h(X) - \IE h(Z)},
\ee
that is, to estimate the error when we replace $X$ in $\IE h(X)$ by $Z$. If the
error in \eq{1} is small irrespective of the detailed properties of $X$ and~$Z$
then we will attribute to the function $h$ a certain degree of
\emph{universality}, which means that the expected value only depends on certain
basic characteristics of $X$ and $Z$, such as the first few moments. 

Of particular interest is the case where $Z$ is a Gaussian vector having the
same (or a similar) covariance structure as $X$, and probably the most prominent
occurrence of such universality is the central limit theorem. If $X$ is a random
vector, such that the $X_i$ are independent of each other, centred and scaled
such that $\sum_i \Var X_i = 1$, and $Z$ is a centred Gaussian vector with
uncorrelated coordinates having the same variances as those of $X$, then it is
well known that \eq{1} is small for functions of the form
\ben								\label{2}
  h(x) = g\bbbklr{\sum_{i=1}^n x_i},
\ee
where $g:\IR\to\IR$ is not too irregular. A common heuristic says that the
central limit theorem will also hold if independence is replaced by some form of
``weak'' dependence, and, furthermore, it can be expected that in many cases
\eq{1} will be small for more general functions than \eq{2}. Thus, in terms of
dropping independence and considering more general functions than \eq{2},
universality often can be observed beyond the standard setting of the central
limit theorem.

Even if the vector $X$ is such that $\sum_i X_i$ does not satisfy the central
limit theorem, we can consider \eq{1} for functions more general than~\eq{2}.
Let, for example, $\xi_i$ be the number of balls that end up in the $i$th urn,
when a fixed number of balls $m$ is distributed independently among $n$ urns.
Clearly, $\sum_i \xi_i = m$, respectively, $\sum_i X_i = 0$ if the $X_i$ are
the centered and properly scaled $\xi_i$. Although these sums do not satisfy a
central limit theorem, it is nevertheless possible to give
informative bounds on \eq{1}. \cite{Dembo1996} and \cite{Chen2010b} considered,
for example, functions of the form $h(x) = g(\sum_{i} \phi(x_i))$ for fixed
functions $g$ and $\phi$, where $\phi$ is non-linear. For other, non-trivial
choices of $h$ we refer to Sections~\ref{sec6} and~\ref{sec7}.

Over the last decades, Stein's method has proved to be a very robust method to
obtain explicit bounds for univariate and multivariate distributional
approximations in cases where $X$ exhibits non-trivial dependencies which are
not of martingale type, but more combinatorial in flavour. Although Stein's
method for the multivariate normal distribution has been successfully
implemented in many places (see \cite{Meckes2009} and \cite{Reinert2009} and
references therein), the
dependence on the dimension of the results obtained so far may give the
impression that the method is not suitable if the dimension grows linearly with
the size of the problem. Indeed, this high-dimensional case has remained
untackled until now. The purpose of this article is to close this gap.

It is important to note at this point that the type of bounds that we will
obtain will generally not imply that the marginal distributions of the
individual coordinates will converge to a normal distribution. That is, the aim
is not to prove convergence to a multivariate normal distribution. In the
already mentioned example of classic occupancy, if the number of balls and urns
are of the same order, then $\xi_i$ will converge to a Poisson distribution with
mean being equal to the limiting ratio $\lim m/n$. Bounds on \eq{1} will only be
informative if they are smaller than the fluctuation of $h(X)$, that is, if the
bounds are smaller than $\IE\abs{h(X)}$ (assuming here without loss of
generality that $\IE h(Z) = 0$), which is an obvious upper bound on~\eq{1}. The
bounds that we obtain for functions $h$ that concentrate only on a few
coordinates will typically have the same order as $\IE\abs{h(X)}$ and hence will
not---and often cannot---be informative.

The remainder of the article is organised as follows. In the next section we
will first discuss the key tools used in this article, in particular the
fundamental idea of using interpolation to estimate \eq{1}, the Gaussian
integration by parts formula and multivariate Stein couplings, leading to our
main result, Lemma~\ref{lem1}. In Section~\ref{sec2} we will then give some
abstract and more concrete examples of Stein couplings, ranging from the
independent case to more sophisticated dependencies. In Section~\ref{sec5} we
will discuss various applications.

\section{The key lemma}\label{sec1}

An old idea to compare two quantities of interest is to find an interpolating
sequence between them and to estimate the error ``along the way'' of the
interpolation using the derivatives of~$h$ (paraphrasing \cite{Talagrand2010} on
``Gaussian interpolation and the smart path method''). One of the earliest
encounters of this idea is Lindeberg's method of telescoping sums. Define the
interpolating sequence
\ben								\label{3}
  Y(i) = (X_1,\dots,X_i,Z_{i+1},\dots,Z_n),
\ee
and write
\ben								\label{4}
  \IE h(X) - \IE h(Z) = \sum_{i=1}^n
	\IE\bklg{h\bklr{Y(i)}-h\bklr{Y(i-1)}};
\ee
one can now bound the right hand side of \eq{4} using Taylor expansion; this
idea has been successfully implemented by \cite{Rotar1973},
\cite{Chatterjee2007a}, \cite{Mossel2010} and \cite{Tao2010b} and surely by
other authors. One of the important consequences of this approach is apparent
when we look at \eq{3}: it forces us to treat the coordinates of $X$ in an
ordered way. If the components of $X$ are independent or, more generally, a
martingale difference sequence, then this is of course desirable, and, indeed,
quite a few central limit theorems for martingales are based upon \eq{4} (see
e.g.\ \cite{Bolthausen1982a}, \cite{Grama1997} and \cite{Rinott1999}). And even
if no such structure is apparent in the problem, one can sometimes arrange $X$
such that it will be close enough to a martingale difference sequence. 

This approach, however, is not entirely satisfying. Often the martingale
structure is ``artificial'' and one would like to make use of a more natural
dependence structure in~$X$, instead (rates of convergences being another reason
to avoid martingales). And in some cases, one may have difficulties to linearise
the problem at all.

A key difference in Stein's method is to chose an interpolating sequence that,
in contrast to Lindeberg's telescoping sum, treats the components of $X$
\emph{symmetrically}. Note that \eq{3} essentially interpolates ``along the
coordinate axes'' and the order of the axes determines the linearisation of the
problem. Instead, we will interpolate between $X$ and $Z$ in a way that will
linearly interpolate between the matrices $XX^t$ and~$ZZ^t$. This approach is
well-known as \emph{Gaussian interpolation} and independently developed by
\cite{Slepian1962} and \cite{Stein1972}, although the technique used by Stein
looks very much different from what is commonly referred to as Gaussian
interpolation (the interpolation is ``hidden'' in the solution to the so-called
\emph{Stein equation}). 

Gaussian interpolation has become popular in many areas; \cite{Talagrand2010}
gives a good account of the key idea, in particular in the context of statistical 
mechanics (where Gaussian interpolation is referred to as \emph{smart path method}). 
The method is a key ingredient in the rigorous proof of the \emph{Parisi formula} by \cite{Talagrand2006}. 
It is also an important tool to prove universality in the bulk of eigenvalues for Wigner 
random matrices with matrix entries following so-called \emph{Gaussian divisible distributions}. 
 The generalisation from these special distributions to the general case, however, 
uses Lindeberg's idea of telescoping sums; see
\cite{Johansson2001} and \cite{Erdos2010}.

Now, assume that $X$ and $Z$ are independent and define the interpolating
sequence $Y_t=\sqrt{t}X+\sqrt{1-t}Z$, $0\leq t\leq1$. Note that, if $\IE X=\IE
Z=0$, then $\IE Y_t = 0$ and, if $\Cov(X)=\Cov(Z)$, then $\Cov(Y_t) = \Cov(X)$
for all $t$ (which may serve as an explanation why this particular $Y_t$ is a
good choice). With $h_i$ being the partial derivative in the $i$th coordinate,
we can write
\besn							\label{5}
  \IE h(X) - \IE h(Z) & = \int_0^1 \frac{\partial}{\partial t}\IE
h(Y_t)dt\\
    & =\frac12 \int_0^1 \IE\bbbklg{
    \frac{1}{\sqrt{t}}\sum_i X_i h_i\klr{Y_t}
    -\frac{1}{\sqrt{1-t}}\sum_i Z_i h_i\klr{Y_t}
    }dt
\ee
(differentiation in \eq{5} corresponds to taking differences in \eq{4} and
integration replaces summation, but this is only a technical difference). One
can easily see that, on the right hand side of \eq{5}, the coordinates are
treated symmetrically. The result obtained by \cite{Slepian1962} (known as
\emph{Slepian's Lemma}) is valid under the assumption that $X$ and $Z$ are
centred Gaussian vectors
having a different covariance structure. In this case, the Gaussian integration
by parts formula 
\ben							\label{6}
    \IE\bklg{Z_i h_i(Z)} = \sum_{j=1}^n \Cov(Z_i,Z_j)\IE h_{ij}(Z)
\ee
can be used to estimate the error on the right hand side of \eq{5} in terms of
the covariances. \cite{Stein1972}, on the other hand, considered the univariate
case, but where $X$ is not Gaussian. Although \eq{6} can still be used for $Z$,
it needs to be replaced by an approximate version of \eq{6} for $X$. 

In order to formalise this approximate version of the Gaussian integration by
parts formula, we will make use of a multivariate generalisation of \emph{Stein
couplings}, which were introduced by \cite{Chen2010b} in the univariate case,
and then give more concrete constructions later on. Throughout this article
summations will always range from $1$ to $n$ unless otherwise stated.

\begin{definition} Let $(X,X',G)$ be a triple of $n$-dimensional random vectors.
We say that the triple is a Stein coupling if, for any smooth enough function
$f:\IR^n \to \IR$, we have
\ben							\label{7}
   \IE \sum_i X_i f_i(X) = \IE\sum_{i}G_i \bklr{f_i(X') -  f_i(X)}
\ee
whenever the involved expectations exist.
\end{definition}

\begin{remark} If $(X,X',G)$ is a Stein coupling, it follows from the definition
that
\ben							\label{8}
  \IE X_i = 0,\qquad \IE(G_iD_j+G_jD_i) = 2\Cov(X_i,X_j),
\ee
for all $i$ and $j$, where we let $D=X'-X$ throughout this article (apply \eq{7} to the functions $f(x)=x_i$ and $f(x)=x_ix_j$, respectively). If \eq{7}
is replaced by the stronger condition that
\ben                                                    \label{9}
   \IE \bklg{X_i f_i(X)} = \IE \bklg{G_i \bklr{f_i(X') -  f_i(X)}}
\ee
for all $i$, then 
\ben                                                    \label{9b}
  \qquad \IE(G_iD_j) = \IE(G_jD_i )= \Cov(X_i,X_j),
\ee
for all $i$ and $j$.
\end{remark}

Equation \eq{7} is the key condition to obtain an approximate Gaussian
integration by parts formula: if $X$ and $X'$ are close to each other, then the
difference on the right hand side of \eq{7} can be approximated by the
corresponding derivatives, leading to a formula similar to~\eq{6}. Hence, it is
crucial that $X'$ is only a \emph{small perturbation} of~$X$.

The following result, although not difficult to prove, is crucial for our
approach. On one hand, it measures how closely $X$ satisfies the Gaussian
integration by parts formula and, on the other hand, also compares the
covariances of $X$ and~$Z$ (which in this article we will mostly assume to be
the same). To make things more transparent, we keep everything explicit in terms
of the function~$h$, instead of using the usual approach via Stein equation and
its solution.

Unless otherwise stated, we will assume throughout this article that
\ben								\label{10}
  \IE X = 0,
  \qquad
  \Var X_i=\sigma^2_i, 
  \qquad
  \IE\abs{X_i}^3=\tau^3_i < \infty,
  \qquad 
  \-\tau = \sup_{i}\tau_i.
\ee
We will denote by $\Sigma= (\sigma_{ij})_{1\leq i,j\leq n}$ the covariance
matrix of $X$, where $\sigma_{ij} = \IE(X_iX_j)$, and we have $\sigma_{ii} =
\sigma_i^2$.

\begin{lemma}\label{lem1} Let $(X,X',G)$ be a Stein coupling. Let $X''$ and
$\~D$ be $n$ dimensional random vectors and let $S$ be a random $n\times n$
matrix. Define $D=X'-X$ and $D' = X''-X$. Assume that, for all $k$ and $l$,
\ben								\label{11}
  \IE(G_k D_l|X) = \IE(G_k \~D_l|X),
    \qquad \IE( S_{kl}|X) = \sigma_{kl}.
\ee
Let $Z\sim\MVN_n(0,\Sigma)$ be independent of the previous random vectors. Then,
for any three times partially differentiable function $h$,
\besn							\label{12}
  \IE h(X) - \IE h(Z)
    & = \frac12\int_0^1\IE R_1(t)  dt
    - \frac12\int_0^1\int_0^1t^{1/2}\IE R_2(t,s)dsdt\\
    &\qquad + \frac12\int_0^1\int_0^1\int_0^1 s t^{1/2}\IE R_3(t,sr)drdsdt,
\ee
where
\ba
  R_1(t) & =
\sum_{k,l}(G_k\~D_l-S_{kl})h_{kl}(\sqrt{t}X''+\sqrt{1-t}Z),\\
  R_2(t,u) & = \sum_{k,l,m}(G_k\~D_l-S_{kl})D'_m
	  h_{klm}(\sqrt{t}X+u\sqrt{t}D'+\sqrt{1-t}Z),\\
  R_3(t,u) & = \sum_{k,l,m}G_kD_lD_m h_{klm}(\sqrt{t}X+u\sqrt{t}D+\sqrt{1-t}Z),
\ee
provided that\/ $\IE R_i(\cdot)$ exists for $i=1,2,3$. In particular,
\be
  \abs{\IE h(X) - \IE h(Z)}
     \leq \ahalf\sup_t\abs{\IE R_1(t) }
    + \athird \sup_{t,s}\abs{\IE R_2(t,s)}
    + \asixth\sup_{t,s}\abs{\IE R_3(t,s)}.
\ee
\end{lemma}

It seems rather difficult at this point to convey the purpose of all the random
vectors appearing in the lemma. Probably the best way to get an intuition for
such couplings is to go through the different applications given later on; we
also refer to \cite{Chen2010b} for the univariate case, where further examples
are discussed. We note that finding the appropriate random vectors will usually
require some trial and error.

\begin{remark}\label{rem1} Let us make a few comments at this point.

\begin{enumerate}
\item We will use the following simple fact in the applications. If $(X,X',G)$
is a Stein coupling satisfying the stronger condition \eq{9} and if there is a
$\sigma$-algebra $\%F''\supset\sigma(X'')$
such that
\ben                                                    \label{13}
  \IE(G_k\~D_l|\%F'')=\IE(S_{kl}|\%F''),
\ee
then $\IE R_1(t)=0$.
\item Except for the case of local dependence in \ref{sec4}, we will choose
$\~D=D$.
\item The result can be easily extended to include other error terms from the
proof of the lemma under weaker conditions. We will use the following extension
later on. If $(X,X',G)$ is not a Stein coupling, then one can include a measure
of how close \eq{7} is satisfied; with 
\bes
  R_0(t) & = \sum_{k}\bklgl
      X_k h_k(\sqrt{t}X +\sqrt{1-t}Z)-G_k h_k(\sqrt{t}X'+\sqrt{1-t}Z)
       \\[-2ex]
      &\kern16em+G_k h_k(\sqrt{t}X +\sqrt{1-t}Z)\bklgr,
\ee
an additional $\frac12\int_0^1\frac{1}{\sqrt{t}}\IE R_0(t)dt$ appears on
the right hand side of~\eq{12}. 
\item If $(X,X',G)$ is not a Stein coupling, the identities \eq{8} and \eq{9b}
are no longer valid and need to be replaced by
corresponding approximate versions.
\item Note that the difference $\abs{G_k \~D_l - S_{kl}}$ in $R_2(t,u)$ can
usually be estimated by $\abs{G_k \~D_l} + \abs{S_{kl}}$ without changing the
rates of convergence. This is not the case for $R_1(t)$, where more care is
required.
\end{enumerate}
\end{remark}

\begin{proof}[Proof of Lemma~\ref{lem1}] Define the interpolating sequence $Y_t
= \sqrt{t}X + \sqrt{1-t}Z$, $0\leq t\leq 1$. Starting from \eq{5}, and using
\eq{6} and~\eq{7}, we obtain
\besn								\label{14}
  & \IE h(X)-\IE h(Z) 
  = \int_0^1\frac{\partial}{\partial t}\IE h\klr{Y_t}dt\\
  & \quad =  \frac12\int_0^1\IE\bbbklg{
      \sum_k \frac{1}{\sqrt{t}}X_k h_k(Y_t)
      - \sum_k \frac{1}{\sqrt{1-t}}Z_k h_k(Y_t)}dt\\
  & \quad =  \frac12\int_0^1\IE\bbbklg{
      \sum_k \frac{1}{\sqrt{t}}G_k(h_k(Y'_t)-h_k(Y_t))
      - \sum_{k,l} \sigma_{kl} h_{kl}(Y_t)}dt,
\ee
where $Y'_t = \sqrt{t}X' + \sqrt{1-t}Z$. Let us recall the definition of
$R_1(t)$ and introduce two additional error terms:
\bg
  R_1(t) = \sum_{k,l}(G_k\~D_l-S_{kl})h_{kl}(Y''_t),\\
  R_4(t) := \sum_{k,l}(S_{kl}-\sigma_{kl})h_{kl}(Y_t),\qquad
  R_5(t)  := \sum_{k,l}G_k(D_l-\~D_l)h_{kl}(Y_t),
\ee
where $Y''_t = \sqrt{t}X'' +
\sqrt{1-t}Z$. Applying 
\be
  h_k(Y'_t)-h_k(Y_t)
    = \int_0^1 \sqrt{t}\sum_l D_l h_{kl}(Y_t + s\sqrt{t} D)ds
\ee
to \eq{14}, and adding and subtracting the terms from
$R_1(t)$, $R_4(t)$ and $R_5(t)$ yields
\bes
  &\IE h(X)-\IE h(Z) \\
    &\quad=\frac12\int_0^1\IE\bbbklg{
      \int_0^1\sum_{k,l} G_kD_l h_{kl}(Y_t+s\sqrt{t}D)ds - \sum_{k,l}
	\sigma_{kl} h_{kl}(Y_t)}dt\\
    &\quad=\frac12\int_0^1\IE\bklg{R_1(t)+R_4(t)+R_5(t)}dt\\
    & \qquad + \frac12\int_0^1\IE\bbbklg{
      \sum_{k,l} (G_k\~D_l-S_{kl})\bklr{h_{kl}(Y_t)-h_{kl}(Y''_t)}}dt\\
    & \qquad + \frac12\int_0^1\int_0^1\IE\bbbklg{
      \sum_{k,l} G_kD_l\bklr{ h_{kl}(Y_t+s\sqrt{t}D)-h_{kl}(Y_t)}}dsdt.
\ee
Note that, under \eq{11}, $\IE R_4(t)=\IE R_5(t)=0$. Taylor expansion in the
last two lines yields the final result; we refer to \cite{Chen2010b} for a more
detailed exposition of the proof in the univariate case. 
\end{proof}

We now derive general norm bounds from Lemma~\ref{lem1}, along the lines of
\cite{Raic2004}, \cite{Chatterjee2008a} and \cite{Meckes2009}. Let
$\norm{\cdot}$ be a norm on $\IR^n$ and let $\tnorm{\cdot}$ be a norm on
$\IR^{n\times n}$, the space of $n\times n$ matrices. Define the following
measures of smoothness of $h$. For $k\geq 1$, let
\be
  M_k(h) =
    \sup_{x\in\IR^n}
    \sup_{u^{(1)},\dots,u^{(k)}\in\IR^n}
    \sum_{i_1,\dots,i_k=1}^n
    \frac{u_{i_1}^{(1)}\cdots u_{i_k}^{(k)}}{\norm{u^{(1)}}\dots\norm{u^{(k)}}}
        h_{i_1,\dots,i_k}(x),
\ee
and for $k\geq 2$ define
\be
  {\wt M}_k(h) =
    \sup_{x\in\IR^n}
    \sup_{A\in\IR^{n\times n}}
    \sup_{u^{(3)},\dots,u^{(k)}\in\IR^n}
    \sum_{i_1,\dots,i_k=1}^n
    \frac{A_{i_1i_2}u_{i_3}^{(3)}\cdots u_{i_{k}}^{(k)}}
          {\tnorm{A}\,\norm{u^{(3)}}\cdots\norm{u^{(k)}}}
            h_{i_1,\dots,i_k}(x)
\ee
(if $k=2$, the third supremum in the definition of ${\wt M}_k(h)$ is just
ignored). We then have the following straitforward result.
\begin{lemma}\label{lem2} Under the assumptions of Lemma~\ref{lem1}, let
$\%F''$ be a
$\sigma$-algebra with $\sigma(X'')\subset\%F''$. Then, for all $0\leq t,s \leq
1$,
\ba
   \abs{\IE R_1(t)} &\leq {\wt M}_2(h)\IE\bklg{\tnorm{\IE(G\~D^t-S|\%F'')}},\\
   \abs{\IE R_2(t,s)} &\leq
     {\wt M}_3(h)\IE\bklg{\bklr{\tnorm{G\~D^t}+\tnorm{S}}\norm{D'}},\\
   \abs{\IE R_3(t,s)} &\leq
    M_3(h) \IE\bklg{\norm{G}\,\norm{D}^2}.
\ee
\end{lemma}

It is clear from this lemma that the optimal choice of the norms
$\norm{\cdot}$ and $\tnorm{\cdot}$ very much depends on the involved
random vectors and how they are coupled. This, in turn, determines which
functions $h$ are considered smooth enough to yield informative bounds.

Let us fix some notation before we proceed. We denote by
$\norm{\cdot}_\infty$ the supremum norm of functions. For $k\geq 1$ and a
$k$-times partially differentiable function $f:\IR^n\to \IR$, we let
\be
  \abs{f}_k = \sup_{1\leq i_1\leq \dots\leq i_k\leq n}
    \norm{f_{i_1\dots i_k}}_\infty.
\ee
For functions $g:\IR\to\IR$ we will use the notation $\norm{g'}_\infty$,
$\norm{g''}_\infty$,
$\dots$, instead of the equivalent $\abs{g}_1$, $\abs{g}_2$, \dots.

The couplings we construct in this article are such that the random vectors
and matrices in Lemma~\ref{lem2} are small with respect to the $L_1$-norms
\be
  \norm{u}_1 = \sum_{i=1}^n \abs{u_i}\qquad\text{and}\qquad
  \tnorm{A}_1 = \sum_{i,j=1}^n \abs{A_{ij}}.
\ee
It is not difficult to see that with respect to these norms we have
\be
     M_k(h) = {\wt M}_k(h) =\abs{h}_k.
\ee
For this reason, we will directly formulate our results in terms of $\abs{h}_k$.

Note that this is in contrast to the results for multivariate normal 
approximation of \cite{Chatterjee2008a} and \cite{Reinert2009}. There, the
vectors and matrices are typically closer in $L_2$ than
in $L_1$. \cite{Meckes2009} showed that in this case $\abs{\cdot}_k$
is too strong to measure the smoothness of $h$, resulting in suboptimal
dependence on the dimension. Using instead $M_k(h)$ and ${\wt M}_k(h)$ with
respect to the $L_2$-norms, \cite{Meckes2009} showed that the dependence on the
dimension can be
substantially reduced.

\begin{remark}\label{rem2}
One may be interested in comparing the distributions of $f(X)$ and $f(Z)$ for
some specific function $f:\IR^n\to \IR$. To this end, choose $h(x)=g(f(x))$ for
$g:\IR\to\IR$. Then, if \eq{1} is small for all three times differentiable
functions $g$, then we can conclude that $f(X)$ and $f(Z)$ are close in
distribution. We record the useful estimates
\bg
  \abs{h}_1  \leq \abs{f}_1 \norm{g'}_\infty,\qquad
  \abs{h}_2 \leq \abs{f}_2 \norm{g'}_\infty+\abs{f}_1^2\norm{g''}_\infty,\\
  \abs{h}_3 \leq \abs{f}_3\norm{g'}_\infty+3 \abs{f}_1\abs{f}_2
\norm{g''}_\infty
	+\abs{f}_1^3\norm{g'''}_\infty.
\ee
\end{remark}

\begin{remark} \label{rem3}
A particular function of interest is
\be
  f(x) = \log\sum_{p=1}^m e^{\beta y^{(p)}(x)}
\ee
for functions $y^{(p)}:\IR^n\to \IR$, $1\leq p \leq m$. Define $\gamma_k =
\sup_p \abs{y^{(p)}}_k$; it is straightforward  to check that
\be
    \abs{f}_1 \leq \beta\gamma_1,
	\quad	
	\abs{f}_2 \leq \beta\gamma_2 
	+ 2\beta^2\gamma_1^2, 
	\quad 
	\abs{f}_3  \leq \beta\gamma_3
	    + 6\beta^2 \gamma_1\gamma_2
	    + 6\beta^3 \gamma_1^3.
\ee

\end{remark}

\section{Couplings}\label{sec2}

Many of the Stein couplings discussed by \cite{Chen2010b} can be adapted to the
multivariate case: exchangeable pairs, size-biasing, local dependence, etc.
Instead of generalising all of them here (which will be done elsewhere with
emphasis on multivariate normal approximation for fixed dimension) we only go
through a few of them explicitly and instead present some other couplings not
discussed by \cite{Chen2010b}. 

\subsection{A theoretical result}\label{sec3}

One may wonder if, given a pair $(X,X')$ with $\IE X=0$, there exists a
$G$ to make the triple
$(X,X',G)$ a Stein coupling. This question has been answered by \cite{Chen2010b}
for the univariate case, but the construction given there can also be used in
the multivariate setting. Let $\%F=\sigma(X)$ be the $\sigma$-algebra induced by
$X$ and let $\%F'=\sigma(X')$. Define formally the sequence
\be
  G = -X + \IE(X|\%F') - \IE(\IE(X|\%F')|\%F)
      + \IE(\IE(\IE(X|\%F')|\%F)|\%F') - \cdots.
\ee
If the sequence converges absolutely in each coordinate, then this will make
$(X,X',G)$ a Stein coupling. Indeed, $\IE(G|\%F) = -X$ and $\IE(G|\%F')=0$ so
that \eq{7} is satisfied. 

To motivate the choice of $G$ used in the next few settings, consider the case
where the coordinates of $X$ are independent. Let $I$ be uniformly distributed
on $\{1,\dots,n\}$, independent of all else. Define the vector $X^{(i)}$ by
\be
  X^{(i)}_k = (1-\delta_{ki})X_k,
\ee
where $\delta_{ij}$ is the Dirac delta function. Let $X' = X^{(I)}$; that is,
$X'$ is the vector where we have set a randomly chosen coordinate to $0$. Denote
by $e_i$ the unit vector in direction $i$. Using independence of the
coordinates,
\be	
  \IE(X|\%F') = \IE\bklr{X^{(I)}+e_IX_I\bmid X^{(I)}} = X^{(I)}
\ee
and
\be	
  \IE(X'|\%F) = \frac{1}{n}\sum_{i} X^{(i)} = (1-\anth)X.
\ee
Hence,
\bes
  G & = -X + X' - (1-\anth)X + (1-\anth)X'
      - (1-\anth)^2X  + (1-\anth)^2X'+\dots\\
    & = -e_IX_I - (1-\anth)e_IX_I -(1-\anth)^2 e_IX_I -  \dots = -n e_IX_I.
\ee

\subsection{Independent coordinates}

In order to illustrate the method in a simple setting, we start with independent
coordinates using $(X,X',G)$ derived in the previous section.

\begin{theorem}\label{thm1} 
Let $X$ be as in \eq{10} and assume the coordinates of $X$ are independent. If
$Z$ is a vector of independent centred Gaussian random variables with the same
variances as $X$, then
\be
  \abs{\IE h(X) - \IE h(Z)}
      \leq \frac{5}{6} \sum_{i}\tau_i^3\norm{h_{iii}}_\infty.
\ee
\end{theorem}

\begin{proof}
Let $G_i = -n\delta_{iI}X_i$ and $X' = X'' = X^{(I)}$, hence $D_i = D_i' =
-\delta_{iI}X_i$. Let $\~D=D$ and $S_{ij}=n\sigma_i^2\delta_{iI}\delta_{jI}$. It
is easy to see that $(X,X',G)$ is a Stein coupling satisfying the stronger
condition \eq{9}, that \eq{11} is satisfied
and that \eq{13} holds with $\%F'' = \sigma(X'',I)$; hence $\IE R_1(t)=0$. The
following estimates are immediate:
\ba
  \abs{R_2(t)} 
    &\leq \sum_{i}
\norm{h_{iii}}_\infty\bklr{\sigma^2_i\IE\abs{X_i}+\IE\abs{X_i}^3}
     \leq 2 \sum_{i} \norm{h_{iii}}_\infty\IE\abs{X_i}^3 ,\\
  \abs{R_3(t)} 
    &\leq \sum_{i} \norm{h_{iii}}_\infty \IE\abs{X_i}^3.
\ee
Lemma \ref{lem1} concludes the theorem.
\end{proof}

Using Lindeberg's telescoping sum and Taylor expansion, and noting
that the first two moments of $X$ and $Z$ match, one easily obtains
\bes
  &\abs{\IE h(X) - \IE h(Z)}\\
      &\qquad \leq \frac{1}{6}
      \sum_{i}(\IE\abs{X_i}^3+\IE\abs{Z_i}^3)\norm{h_{iii}}_\infty
     \leq \frac{(1+\sqrt{8/\pi})}{6}     
	\sum_{i}\tau_i^3	  \norm{h_{iii}}_\infty.
\ee

Not surprisingly, the constants obtained via Stein's method are larger for the
case of independent random variables. However, applications with dependencies is
the main purpose of using Stein's method.

\subsection{Weak dependence}

A simple way to measure how much a single coordinate $X_i$ is influenced by the
other coordinates is to look at the fluctuation of the conditional mean and
variance of $X_i$. To this end, let $X$ be as in \eq{10} and define $X^{(i)}$
as in Section~\ref{sec3}. Furthermore, let
\be
  \mu_i\bklr{X^{(i)}} = \IE\bklr{X_i\!\bmid\!X^{(i)}},
  \qquad	
  \sigma_i^2\bklr{X^{(i)}} = \Var\bklr{X_i\!\bmid\!X^{(i)}}.
\ee
Then we have the following.

\begin{theorem}\label{thm2} Let $X$ be as in \eq{10} and let
$Z\sim\MVN_n(0,\Sigma)$. Then
\bes
  & \babs{\IE h(X) - \IE h(Z)} \\
  & \quad\leq
    \sum_{i}\norm{h_i}_\infty\IE\abs{\mu_i(X^{(i)})} + 
    \frac12\sum_i\norm{h_{ii}}_\infty\bklr{\IE \mu_i(X^{(i)})^2
	  +\IE\abs{\sigma_i^2(X^{(i)})-\sigma_i^2}}\\
  &\qquad + \frac{5}{6}\sum_{i}\tau_i^3\norm{h_{iii}}_\infty
\ee
\end{theorem}

\begin{proof} Define $G$, $X'$, $X''$, $S$ as in the proof of
Theorem~\ref{thm1}; the error terms $R_2$ and $R_3$ can be
bounded in the same way. As $(X,X',G)$ is not
necessarily a Stein coupling, we need the additional error term
\be
  \IE R_0(t) = \IE\sum_{i}X_i h_i(\sqrt{t}X^{(i)}+\sqrt{1-t}Z)
   = \IE\sum_{i}\mu_i(X^{(i)}) h_i(\sqrt{t}X^{(i)}+\sqrt{1-t}Z)
\ee
(see Remark~\ref{rem1}). Furthermore,
\bes
  \IE R_1(t) & = \IE\sum_{i}(X_i^2-\sigma_i^2)
h_{ii}(\sqrt{t}X^{(i)}+\sqrt{1-t}Z)\\
  & = \IE\sum_{i}((X_i-\mu_i(X^{(i)}))^2-\sigma_i^2-\mu_i(X^{(i)})^2)
h_{ii}(\sqrt{t}X^{(i)}+\sqrt{1-t}Z)
\ee
This easily leads to the final bound. 
\end{proof}

Note that if the $X_i$ are independent, Theorem~\ref{thm2} reduces to
Theorem~\ref{thm1}. \cite{Gotze2006} assumed that $\mu_i(X^{i})=0$
almost surely to obtain convergence rates to the semi-circular law in random
matrix theory under such dependence.

\subsection{Constant sum and symmetry}

Recall the classic occupancy problem from the introduction. The sum of the
vector that describes the number of balls in each urn is equal to the
total number of balls and hence, itself, does not satisfy a central limit
theorem. This motivates us to consider general centered vectors $X$ that satisfy
\ben								\label{15}
  \sum_{i}^n X_i = 0
\ee
almost surely.

To apply our method, we will need to make more assumptions. A random vector
$X=(X_1,\dots,X_n)$ is called exchangeable if its distribution is invariant
under permutation of the coordinates.  Note that \eq{15} implies
$\sum_{j}\sigma_{ij} = 0$ for each~$i$, and combined with exchangeability, we
therefore have
\ben                            \label{16}
  \sigma_{ij} = -\frac{\sigma_1^2}{n-1}
\ee
for all $i\neq j$. 
  
\begin{theorem}\label{thm3} Let $X$ be an exchangeable random vector satisfying
\eq{10} and let $Z\sim \MVN_n(0,\Sigma)$. Then
\ben								\label{17}
    \abs{\IE h(X) - \IE h(Z)} \leq
\abs{h}_2\bbbkle{\Var\bbbklr{\sum_{i} X_i^2}}^{1/2}
  + 16\abs{h}_3n\tau_1^3.
\ee
\end{theorem}

\begin{remark} Note that the theorem can also be applied if $X$ is not
exchangeable, but $h$ symmetric instead, that is if $h(x)$ remains the same
under any permutation of the coordinates. In that case,
Theorem~\ref{thm3} can be applied to the randomly permuted $X$. Note that
$\tau_1^3$ is then replaced by $n^{-1}\sum_{i}\tau_i^3$ for the final result.
\end{remark}

\begin{proof} [Proof of Theorem~\ref{thm3}]
For $x\in\IR^n$, let $x^{ik}\in\IR^n$ be
the vector obtained by interchanging the $i$th and $k$th coordinate of $x$ (if
$i=k$
then $x^{ik}=x$). Note that due to exchangeability,
\ben							\label{18}
  \IE\bklg{\phi(X_i)h_{i}(X^{ik})} = \IE\bklg{\phi(X_k)h_i(X)},
\ee
for any function $\phi$ for which the expectations exist.
Furthermore, for $(i,j,k,l)\in [n]^4$ with 
\ben								\label{19}
	i=k\iff j=l,
\ee
denote by $x^{ijkl}$ a permutation of $x$ such that
\ben							\label{20}
	\IE\bklg{\phi(X_j,X_l) h_{ik}(X)} 
        = \IE\bklg{\phi(X_i,X_k) h_{ik}(X^{ijkl})}                      
\ee
for all functions $\phi$ for which the expectations exist.  Note that this
permutation can be defined independently of~$x$ and $h$: if $X$ is
exchangeable, keep $[n]\setminus\{i,j,k,l\}$ fixed, map $j\mapsto
i$ and $l\mapsto k$ and map the remaining numbers among each other in
any arbitrary, but fixed way. Let $(I,J,K,L)$ be distributed on $[n]^4$, such
that $(I,J,K,L)$ is uniform on $[n]^3$ and, given $(I,J,K)$, $L$ is uniform on
$[n]\setminus\{J\}$ if $I\neq K$, and $J=L$ if $I = K$; hence, $(I,J,K,L)$
satisfies~\eq{19}. Define
\be
	X' := X^{I K},
  \qquad 
	X'' := X^{I J K L},  
\ee
and
\be
    G_k = -n\delta_{kI}X_k,\qquad {\~D}_k = D_k, \qquad
    S_{kl} = n^2\delta_{kI}\delta_{lK}\sigma_{kl};
\ee
note that
\be
    D_l = \delta_{lI}(X_{K}-X_{I}) +\delta_{lK}(X_{I}-X_{K}),
    \quad
    D'_l = \sum_{m\in\{I,J,K,L\}}\delta_{lm}(X''_{l}-X_l).
\ee
Fix $t$ and let, for notational convenience, $f_\cdot(x) = \IE
h_\cdot(\sqrt{t}x+\sqrt{1-t}Z)$, where $\cdot$ stands for $i$, $ij$ or $ijk$.
Clearly, $\IE\bklg{G_k f_k(X)} = \IE\bklg{X_k f_k(X)}$. Using exchangeability of
$X$, we can use \eq{18} to obtain
\bes
  \IE\sum_{k} G_k f_k(X')& = -n\IE\klg{X_{I} f_{I}(X^{IK})}
      = -n\IE\klg{X_{K} f_{I}(X)}\\
     & = -\frac{1}{n}\IE\sum_{i,k}X_k f_{i}(X)
     =0.
\ee
Hence, $(X,X',G)$ is a Stein coupling. Now,
\bes
  & \IE R_1(t) \\
   & =\IE\sum_{k,l}(S_{kl}-G_kD_l)f_{kl}(X'')\\
    & =\IE\sum_{k,l}\bkle{
    n^2\delta_{kI}\delta_{lK}\sigma_{kl}+\delta_{kI}n X_k
      \bklr{\delta_{lI}(X_{K}-X_{l}) +\delta_{lK}(X_{I}-X_l)}}f_{kl}
(X'')\\
    & =n^2 \IE \sigma_{I K}f_{I K}(X'')+n\IE X_{I}
      (X_{K}-X_{I})f_{{I}{I}}(X'') +
	  n\IE
	    X_{I}(X_{I}-X_{K})f_{I K}(X'').
\ee
Using exchangeability and \eq{16},
\bes
  n^2 \IE \klg{\sigma_{I K}f_{I K}(X'') }
  & = n^2 \IE \klg{\sigma_{JL}f_{IK}(X)} \\
  & = \frac{1}{n}\IE\sum_{i,j}\sigma_{jj}f_{ii}(X)
      + \frac{1}{n(n-1)}
	\IE\sum_{i,j,k\neq i, l\neq j}\sigma_{jl}f_{ik}(X)\\
  & = \sigma_1^2\IE\sum_{i}f_{ii}(X)
      - \frac{\sigma^2_1}{n-1}
	\IE \sum_{i,k\neq i}f_{ik}(X)\\
\ee
Furthermore, using \eq{20} and \eq{15},
\bes
  n\IE \klg{X_{I}       (X_{K}-X_{I})f_{{I}{I}}(X'')}
  &= n\IE\klg{ X_{J}
      (X_{L}-X_{J})f_{II}(X)} \\
  &= \frac{1}{n^2}\IE\sum_{i,j,l}X_{j}
      (X_{l}-X_{j})f_{ii}(X) \\
  &= -\frac{1}{n}\IE\sum_{j}X_{j}^2\sum_{i}f_{ii}(X) \\
\ee
and
\bes
  &n\IE\klg{ X_{I}(X_{I}-X_{K})f_{I K}(X'')}\\
  &\qquad = n\IE\klg{ X_{J}(X_{J}-X_{L})f_{I K}(X)}\\
  &\qquad = \frac{1}{n^2(n-1)}\IE\sum_{i,j,k\neq i,l\neq j}
X_{j}(X_{j}-X_{l})f_{ik}(X)\\
  &\qquad = \frac{1}{n^2}\IE\sum_{i,j,k\neq i}
X_{j}^2f_{ik}(X)
  - \frac{1}{n^2(n-1)}\IE\sum_{i,j,k\neq
i,l\neq j}
X_{j}X_{l}f_{ik}(X)\\
  &\qquad = \frac{1}{n^2}\IE\sum_{j}X_j^2\sum_{i,k\neq i} f_{ik}(X)
  - \frac{1}{n^2(n-1)}\IE\sum_{j,l\neq
j}X_{j}X_{l}\sum_{i,k\neq i}f_{ik}(X)\\
  &\qquad = \frac{1}{n^2}\IE\sum_{j}X_j^2\sum_{i,k\neq i} f_{ik}(X)
  + \frac{1}{n^2(n-1)}\IE\sum_{j}X_{j}^2\sum_{i,
k\neq i}f_{ik}(X)\\
  &\qquad = \frac{1}{n(n-1)}\IE\sum_{j}X_{j}^2\sum_{i,
k\neq i}f_{ik}(X),
\ee
where for the last equality we used that
$\frac{1}{n(n-1)}=\frac{1}{n-1}-\frac{1}{n}$. Hence,
\bes
  \abs{\IE R_1(t)}
    & \leq 
\frac{1}{n(n-1)}\bbbabs{\IE\bbbklg{\bbbklr{\sum_{j}X_j^2-n\sigma_1^2 }
      \sum_{i,k\neq i} f_{ik}(X)}}\\
    &\qquad +
\frac{1}{n}\bbbabs{\IE\bbbklg{\bbbklr{\sum_{j}X_j^2-n\sigma_1^2}
      \sum_{i} f_{i}(X)}}\\
    & \leq 2\abs{h}_2\bbbkle{\Var\bbbklr{\sum_{i} X_i^2}}^{1/2}.
\ee
This gives the first part of the result. 
Now,
\bes
  \IE R_2(t,u) & = \IE \sum_{k,l,m}(S_{kl}-G_k\~D_l)D'_m
	  f_{klm}(X+uD')\\
  & = n^2\IE\sum_{m\in\{I,J,K,L\}}\sigma_{IK}(X''_m - X_m) 
f_{IKm}(X+uD')\\
      & \quad-n\IE\sum_{l\in
      \{I,K\},m\in\{I,J,K,L\}}X_{I}(X'_l-X_l) (X''_m - X_m)
	  f_{Ilm}(X+uD')
\ee
hence
\be
  \abs{\IE R_2(t,u)} \leq
    8\abs{h}_3\tau_1\sum_{i,j}\abs{\sigma_{ij}}+32
    \abs{h}_3n\tau_1^3
    \leq
    8\abs{h}_3\tau_1 n\sigma^2_1+32
    \abs{h}_3n\tau_1^3.
\ee
Similarly,
\bes
  \IE R_3(t,u) & = \sum_{k,l,m}G_kD_lD_m h_{klm}(X+uD)\\
    & = n\IE \sum_{l\in\{I,K\},m\in\{I,K\}}X_{I} (X'_l-X_l)(X'_m-X_m)
h_{klm}(X+uD),
\ee
hence
\be
  \abs{\IE R_3(t,u)}\leq 16\abs{h}_3n\tau^3_1 .\qedhere
\ee
\end{proof}

\subsection{Local dependence}\label{sec4}

Stein couplings to handle local dependence has already been discussed by
\cite{Chen2010b}, based on similar decompositions that appeared in many other
places; we refer to the more detailed discussion in \cite{Chen2010b}.
In particular, multivariate normal approximation for sums of locally dependent
random vectors was considered by \cite{Rinott1996} and \cite{Raic2004}.

Let $X=(X_1,\dots,X_n)$ be as in \eq{10}. 
Assume that, for each $i\in[n]:=\{1,\dots,n\}$, there is a subset
$A_i\subset[n]$ such that 
$X_{A^c_i}$ and $X_i$ are independent. Assume further that for each
$i\in[n]$ and $j\subset A_i$ there is a subset $B_{ij}\subset[n]$ such that
$A_i\subset B_{ij}$ and $X_{B^c_{ij}}$ is independent of $(X_i,X_j)$. Central
limit theorems for sums of random variables satisfying this refined
version of local dependence were analyzed in detail by \cite{Barbour1989}.

\begin{theorem}\label{thm4} Let $X$ as above. Let $Z\sim\MVN_n(0,\Sigma)$.
Then, for any three times partially
differentiable function $h$,
\bes
  \abs{\IE h(X) - \IE h(Z)} &\leq 
      \frac13\sum_{i}\sum_{j\in A_i}\sum_{k\in	B_{ij}}
	   \bklr{\abs{\sigma_{ij}}\IE\abs{X_k}
		+ \IE\abs{X_iX_jX_k}}\norm{h_{ijk}}_\infty\\
      &\qquad +\frac16\sum_{i}\sum_{j,k\in A_i}
          \IE\abs{X_iX_jX_k}\norm{h_{ijk}}_\infty
       \leq \frac56\-\tau^3n\eta\abs{h}_3
\ee
where $\eta = \sup_{i}\sum_{j\in A_i}\abs{B_{ij}}$.
\end{theorem}
\begin{proof} Let $I$ be uniform on $[n]$ and, given $I$, let $J$ be uniform on
$A_I$. Define the vectors $X'$, $X''$, $G$ and $\~D$ and the matrix $S$ as
\bg
    G_k = -\delta_{kI}n X_k,\qquad
    X'_k = \I[k\not\in A_I]X_k,\qquad
    X_k'' = \I[k\not\in B_{IJ}]X_k,\\
    S_{kl} = n\abs{A_I}\delta_{kI}\delta_{lJ}\sigma_{kl},\qquad
    \~D_k = -\abs{A_J}\delta_{kJ} X_k.
\ee
Note that $X'$ is independent of $G$, which makes $(X,X',G)$ a Stein coupling
satisfying the stronger condition \eq{9}, similarly as for the independent
case. Furthermore, with $\%F''=\sigma(X'',I)$, \eq{13} holds and therefore $\IE
R_1(t)=0$. The final bound follows now easily from
Lemma~\ref{lem1}. 
\end{proof}

As we can see from the case of the CLT, where $h(x) =
g(\sum_i x_i)$, the typical scaling of $X$ is such that
$\-\tau^3 \asymp n^{-3/2}$. With this scaling, a ``typical'' function $h$ will
have the property that $\IE\abs{h(X)}\asymp  1 $, whereas the bound of
Theorem~\ref{thm4} is of order $\bigo(n^{-1/2})$.

Note that an $m$-dependent sequence is a special case of local dependence: we
have
$\abs{A_i} = 1+2m$ and $B_{ij} \leq 1+3m$. However, the crucial aspect here is
that the exact
structure of the dependence is only important in terms of the size of
$A_i$ and $B_{ij}$. Any graph with maximal degree $m$ that describes the
dependence structure of $X$ (that is, two subsets of vertices are independent if
there is no edge between them) will have the upper bounds $\abs{A_i}\leq 1+m$
and $\abs{B_{ij}}\leq 1+2m$. In that case, 
\ben								\label{21}
  \eta \leq 2(m+1)^2.
\ee

\section{Applications}\label{sec5}

In this section, we present two different types of applications.
First, we consider concrete functions $h$, for which we determine
under what kind of dependencies \eq{1} is small. If we can control the
first three derivatives of $h$, then we can analyse the universality of the
given $h$ with respect to dependence, for example for the different settings of
the previous section. The first two applications below are of this type. We
analyse universality with respect to local dependence only, but it is clear that
many of the other settings can be used instead. In the case of local dependence,
we are interested in how big the ``neighbourhoods'' $A_i$ and $B_{ij}$ are
allowed to become while keeping the bounds on \eq{1} small enough. We use $\eta$
from Theorem~\ref{thm4} as a simple measure of neighbourhood size, and hence
dependence. These applications are closely related to \cite{Chatterjee2005a}.
Whereas in the first application of the SK-model the dependence enters in a
straightforward way, in the second application of last passage percolation on
thin rectangles, an certain optimisation step has to be recalculated, including
the measure of dependence~$\eta$.

As a second type of application, we can consider more concrete vectors $X$, for
which we want to show that \eq{1} is small for a large class of functions~$h$.
In this situation, the structure of the dependence of $X$ will either fit into
one of the abstract settings of the previous section (this is the case for
classic occupancy), or else, one has to construct a Stein coupling from
scratch; the latter is the case for the Curie-Weiss model.

\subsection{Environment with dependencies in the Sherrington-Kirkpatrick spin
glass
model} \label{sec6}

Consider the $N$-spin system $\{-1,1\}^N$. To each configuration
$\sigma\in\{-1,1\}^N$ we assign the (random) Hamiltonian
\be
  H_N(\sigma) = \frac{\beta}{\sqrt{N}}\sum_{i<j}\xi_{ij}\sigma_i\sigma_j,
\ee
where $\xi=(\xi_{ij})_{1\leq i< j\leq n}$ is a family of random variables,
which we call the \emph{environment}. Given the environment $\xi$, we assign to
each $\sigma$ the probability
\be
  \IP^\xi_{N}(\sigma) = \frac{e^{ H_N(\sigma)}}{Z_N(\beta,\xi)},
\ee
where
\be
  Z_N(\beta,\xi) = \sum_{\sigma}e^{\beta H_N(\sigma)}.
\ee
Let
\be
  p_N(\beta) = \frac{1}{N}\IE\log Z_N(\beta,\xi).
\ee
It was proved by \cite{Talagrand2006} that $p_N(\beta) \to p_\infty(\beta)$,
the solution of the \emph{Parisi formula}, if the $\xi_{ij}$ are independent
standard Gaussians. \cite{Carmona2006} showed that the same limit
holds if the Gaussians are replaced by independent copies of any random
variable $\xi$ with $\IE \xi = 0$ and $\IE\abs{\xi}^3< \infty$.
We shall extend this results to dependent environments. To this end define
\be
  \~Z_n(\beta,\xi) 
    = \IE^\xi\bklg{e^{\beta\sum_{i=1}^n Y_i\xi_i}}.
\ee
where $Y_1,\dots,Y_n$ is any family of random variables such that $Y_i$
only takes finitely many values and $\abs{Y_i}\leq 1$ for all~$i$.

\begin{lemma}\label{lem3} Let $\xi = (\xi_1,\dots,\xi_n)$ be a random
environment such that $\IE\xi_i=0$, $\IE\xi_i^2=1$ and $\IE\abs{\xi_i}^3\leq
\-\tau^3<\infty$, satisfying the dependence structure of Theorem~\ref{thm4}. Let
$g\sim\MVN_n(0,\Sigma)$ where $\Sigma$ is the covariance
matrix of $\xi$. Then
\ben						\label{22}
  \abs{\IE\log \~Z_n(\beta,\xi)-\IE\log \~Z_n(\beta,g)}
      \leq 5\beta^3\-\tau^3n\eta.
\ee
\end{lemma}
\begin{proof} Let $h(\xi) = \log \~Z_n(\beta,\xi)$;
it is easy to see from Remark~\ref{rem3} that
\be
  \norm{h_{ijk}}\leq 6\beta^3.
\ee
(note that $\gamma_2=\gamma_3=0$ and $\gamma_1\leq 1$ as $\abs{Y_i}\leq 1$).
Using Theorem~\ref{thm4}, \eq{22} is immediate. 
\end{proof}

The following statement is a direct consequence of Lemma~\ref{lem3} for
$n=N(N-1)/2$ and $\beta$ replaced by $\beta N^{-1/2}$.
  
\begin{theorem}\label{thm5} Assume the environment $\xi$ satisfies the
dependence structure of Theorem~\ref{thm4} with $\eta=\lito(N^{1/2})$ and
$\sigma_{ij}=0$ for $i\neq j$. Then
\be
  \frac{1}{N}\IE\log Z_N(\beta,\xi) \to p_\infty(\beta).
\ee
\end{theorem}

Consider a fixed $m$-regular graph $G$ on the set of
vertices
$V_N=\{(i,j)\,:\,1\leq i<j\leq N\}$. Let $h_{ij}$ be i.i.d.\ centred random
variables with finite third moments. Let 
\be
  \xi_{ij} = \prod_{(k,l)\sim(i,j)} h_{kl}.
\ee
Then it is straightforward to see that these $\xi_{ij}$ are centred and
uncorrelated (note that $\xi_{ij}$ does not contain $h_{ij}$).
Clearly, from \eq{21}, $\eta\leq 2(m+1)^2$ and hence  we can apply
Theorem~\ref{thm5} as long as 
$m= \lito(N^{1/4})$. Noticing that \eq{22} is independent of the underlying
graph, we obtain the following.

\begin{corollary} Let $G_N$ be a sequence of random $m_N$-regular graphs on
$V_N$, where $m_N=\lito(N^{1/4})$. Then, with $\xi$ as
above,
\be
  \frac{1}{N}\IE^{G_N}\log Z_N(\beta,\xi) \to p_\infty(\beta)
\ee
almost surely.
\end{corollary}

\subsection{Last passage percolation for thin rectangles}
\label{sec7}

The following statements about smooth approximation of the maximum function is
well-known (and easy to verify).

\begin{lemma}\label{lem4} Let $m$ be a positive
integer.
For each $y\in \IR^m$, let $f_0(y) = \max\{y_1,\dots,y_m\}$ and
$f_\eps(y) = \eps \log{\sum_{i}e^{y_i/\eps}}$. Then
\be
  0\leq f_\eps(y)-f_0(y)\leq \eps\log m.
\ee
\end{lemma}

Consider functions $y^{(p)}:\IR^n\to\IR$, $p=1,\dots,m$, and let
\ben									\label{23}
  P_x = \max_{1\leq p\leq m}y^{(p)}(x).
\ee

The following theorem is similar to a result obtained by \cite{Chatterjee2005a},
but now includes~$\eta$. To keep the bounds simple we make the stronger
assumption that the functions $y^{(p)}$ are linear, which what we will
need subsequently.

\begin{theorem}\label{thm6} Let $P_x$ be as above with  linear
functions~$y^{(p)}$. Let $X$ be a family of $n$ centred random variables with
finite third moments satisfying the dependence structure as in
Theorem~\ref{thm4}. Let $g:\IR\to\IR$ be three times differentiable. Then, for
$Z\sim\MVN_n(0,\Sigma)$,
\be
  \babs{\IE g(P_X)-\IE g(P_Z)} 
    \leq \bklr{6\norm{g'}_\infty+6\norm{g''}_\infty+\norm{g'''}_\infty}
n^{1/3}\eta^{1/3}\-\tau\gamma_1\log(m)^{2/3}.
\ee
where $\gamma_1 = \sup_{1\leq p\leq m}\abs{y^{(p)}}_1$.
\end{theorem}
\begin{proof}Using the notation of Lemma~\ref{lem4}, define the functions
\be
  h_0(x) = g(f_0(x)) ,
  \qquad
  h_\eps(x) = g(f_\eps(x)).
\ee
Clearly
\be
  \abs{h_0(x)-h_\eps(x)}\leq \norm{g'}_\infty\eps\log m.
\ee
We now use Remark~\ref{rem3}. We clearly have $\gamma_2 =
\gamma_3 = 0$.
Furthermore, using again Lemma~\ref{lem4}, it is easy to check that,
\bes
  \abs{h_\eps}_3 & \leq 
\abs{f_\eps}_3\norm{g'}_\infty+3\abs{f_\eps}_1\abs{f_\eps}_2
    \norm{g''}_\infty+\abs{ f_\eps}_1^3\norm {g''' }_\infty\\
    & \leq\eps^{-2}\gamma_1^3\bklr{6\norm{g'}_\infty+6\norm{g''}_\infty
      +\norm{g'''}_\infty}.
\ee
Thus, using Theorem~\ref{thm4},
\bes
  &\abs{\IE h_0(X) - \IE h_0(Z)} \\
      &\qquad \leq \norm{g'}_\infty\eps\log m + \abs{\IE h_\eps(X) - \IE
h_\eps(Z)}\\
      &\qquad \leq \norm{g'}_\infty\eps\log m 
	 + C\eps^{-2}\-\tau^3n\eta\gamma_1^3
	\bklr{6\norm{g'}_\infty+6\norm{g''}_\infty+\norm{g'''}_\infty}.
\ee
Choosing $\eps = n^{1/3}\eta^{1/3}\log(m)^{-1/3}\-\tau\gamma_1$, we obtain the final
bound.
\end{proof}

Let us apply this result to last passage percolation on thin rectangles along
the lines of \cite{Suidan2006}. Denote by $\pi$ an increasing path from $(1,1)$
to $(N,k)$ on the usual two dimensional lattice, where without loss of
generality $k\leq N$.
Let
\be
	y^{(\pi)}(x)
          = \frac{k^{1/6}}{N^{1/2}}\bbbklr{\sum_{i\in\pi }x_i-2\sqrt{Nk}}
\ee
and let $P_x$ be as in \eq{23}, where the maximum ranges over all increasing
paths~$\pi$. Hence, $P_x$ is the (standardized) longest increasing path between
$(1,1)$ and $(N,k)$, where each lattice point $(i,j)$ contributes $x_{ij}$ to
the length of the path. If $X$ is an i.i.d.\ family of geometric or exponential
random variables, then \cite{Johansson2000} showed that the properly centred and
standardized $P_X$ will converge to $F_2$ (the Tracy-Widom distribution for
Gaussian unitary ensembles) if $k=N$. For independent $X_i$ that are neither
exponentially nor geometrically distributed, the same results is only known for
thin rectangles, that is, for $k$ being of smaller order than~$N$; see
\cite{Bodineau2005}, \cite{Baik2005} and \cite{Suidan2006}. In particular, if
$X_i$ have finite third moments, then $k = \bigo(N^{\alpha})$ for $\alpha<1/7$.
We shall expand this result to locally dependent $X$. If $\eta$ remains bounded,
we recover the same maximal order for $k$ as in the independent case. If $\eta$
grows with $N$, however, the maximal order of $k$ be will be affected.

\begin{corollary}\label{thm7} Let $X=(X_{ij})_{1\leq i\leq N, 1\leq j\leq k}$ be
a collection of $n=Nk$ random variables with mean $0$ and variance $1$,
satisfying the dependence structure of Theorem~\ref{thm4}, and let
$Z\sim\MVN_n(0,\Sigma)$. Then, for any three times differentiable function
$g:\IR \to \IR$,
\bes
  \babs{\IE g(P_X)-\IE g(P_Z)} 
    \leq \frac{C(g,\-\tau) \eta^{1/3}k^{7/6}\log(N)^{2/3}}{N^{1/6}}.
\ee
For some constant $C(g,\-\tau)$. If $\sigma_{ij}=0$ for all $i\neq j$, then the
$P_X$ will converge to $F_2$ if $\-\tau$ remains bounded and if
\be
	k = \lito\bklr{N^{1/7}\log(N)^{-4/7}\eta^{-2/7}}.
\ee
\end{corollary}

\begin{proof}
Clearly, $\gamma_1 = k^{1/6}N^{-1/2}$. Furthermore,
\be
  m = {N+k\choose N}
  \leq \bbbklr{\frac{N}{k}}^{k}\bbbklr{\frac{N+k}{N}}^{N+k}
\ee
As
\bes
	\log(m) & = k(\log(N)-\log(k)) + (N+k)(\log(N+k)-\log(N))\\
			& \leq k\log(N)+2k
\ee
Applying Theorem~\ref{thm7} yields the final bound. 
\end{proof}

\subsection{Classic occupancy}

As mentioned in the introduction, we can obtain bounds on \eq{1} for the classic
occupancy problem. Distribute $m$ balls independently and uniformly among $n$
urns. Let $\xi_i$ be the number of balls in urn~$i$. Then,
$\xi_i\sim\Bi(m,n^{-1})$ and $\sum_i \xi_i = m$, and therefore 
\be
  X_i = \frac{\xi_i-mn^{-1}}{\sqrt{m(1-n^{-1})}}
\ee
satisfies \eq{10} and \eq{15} and in addition $\sum_i \sigma^2_i = 1$. 
 
\begin{theorem} Let $X$ be as above and let $Z\sim\MVN_n(0,\Sigma)$. Then, for
any three times partially differentiable function $h$,  
\be
    \abs{\IE h(X) - \IE h(Z)} \leq
(\abs{h}_2+19\abs{h}_3)\sqrt{\frac{n^2+4mn+6}{mn(n-1)}}.
\ee
\end{theorem}

\begin{proof} We can apply \eq{17}, as $X$ is exchangeable and \eq{15} is
satisfied. It is straightforward to verify that
\be
  \Var X_1^2 = \frac{n^2+2(n-1)(m-3)}{n^2m(n-1)},
  \qquad
  \Cov(X_1^2,X_2^2) = -\frac{n^2-4n-2m+6}{mn^2(n-1)^2}
\ee
(see Lemma~\ref{lem5} below), which implies
\be
  \Var \sum_i X_i^2 = n\Var X_1^2 + \frac{n(n-1)}{2}\Cov(X_1^2,X_2^2)
	 = \frac{n^2+4mn-2m-8n+6}{2mn(n-1)}.
\ee
Furthermore,
\be
  \IE\abs{X_1}^3 \leq \sqrt{\IE X_1^2\IE X_1^4}\leq \sqrt{\frac{n^2+3nm-6n-3
m+6}{mn^3(n-1)}}.
\ee
From this, the final bound follows.
\end{proof}

We record here some identities for the mixed moments of the $\xi_i$, which are
easy to verify and needed in the above calculations. \begin{lemma}\label{lem5}
Let $\xi_1$ and $\xi_2$ be the number of balls when distributing $m$ balls
uniformly and independently among $n$ urns. Then
\ba
  \IE \xi_1^2 & = \frac{m}{n}+\frac{m(m-1)}{n^2},\\
  \IE(\xi_1\xi_2) & = \frac{m(m-1)}{n^2},\\
  \IE(\xi_1^2\xi_2) & = \frac{m(m-1)}{n^2}+ \frac{m(m-1)(m-2)}{n^3},\\
  \IE(\xi_1^2\xi_2^2) & = \frac{m(m-1)}{n^2}+ 2\frac{m(m-1)(m-2)}{n^3}
	 +\frac{m(m-1)(m-2)(m-3)}{n^4}.
\ee 
\end{lemma}

\subsection{Currie-Weiss model in the high-temperature regime}

Consider the $n$-spin system $\{-1,1\}^n$ with Hamiltonian
\be
  H(\sigma) = -\frac{1}{n}\sum_{i<j}\sigma_i\sigma_j.
\ee
To each configuration $\sigma$ assign the probability
\be
	\IP(\sigma) = \frac{e^{-\beta H(\sigma)}}{Z(\beta)},
\ee
where $Z(\beta)$ is the normalising constant. This model is well-known as
Curie-Weiss model; we refer to \cite{Eichelsbacher2010} for a more detailed
discussion of relevant literature. The authors of that article prove in
particular bounds in univariate central limit theorems for the total
magnetisation of this and similar models. Here, instead, we will estimate the
error when we replace all the spins by corresponding Gaussian variables in the
high-temperature regime $\beta<1$; this, in particular, implies the central
limit theorem for the total magnetisation.  

Previous approaches using Stein's method to analyse the magnetisation of such
models make use of exchangeable pairs (\cite{Eichelsbacher2010} and
\cite{Chatterjee2010}) which typically involves resampling a spin conditional on
the other spins. It is worthwhile noting that the Stein coupling we will use
does not require any resampling and, hence, does not form an exchangeable pair. 

To avoid confusion with the notation $\sigma_i$ for the spins, we will use
$s_{ij}$ instead of $\sigma_{ij}$ to denote covariances in what follows.  

\begin{theorem} Let $X_i=n^{-1/2}\sigma_i$ and let $Z\sim\MVN_n(0,\Sigma)$,
where
$\Sigma=(s_{ij})_{1\leq i,j\leq n}$ with 
\be
	s_{ij} = \begin{cases}
		\displaystyle\frac{1}{n}+\frac{\beta}{n^2(1-\beta)},
                                  &\text{if $i=j$,}\\[3ex]
		\displaystyle\frac{\beta}{n^2(1-\beta)},&\text{if $i\neq j$.}\\
		\end{cases}
\ee
Then, for $\beta<1$,
\be
  \abs{\IE h(X) - \IE h(Z)}
  \leq  C_\beta\bbbklr{\frac{\abs{h}_1+\abs{h}_3}{n^{1/2}}+\frac{\abs{h}_2}{n}}
\ee
for some constant $C_\beta$ that only depends on $\beta$.
\end{theorem}

\begin{proof} Define
\be
	m_i = \frac{1}{n}\sum_{j\neq i}\sigma_j,
	\qquad m = \frac{1}{n}\sum_{i}\sigma_i.
\ee
We recall the estimates 
\be
	\IE \abs{m}^k\leq C_\beta n^{-k/2};
\ee
see \cite[Lemma~3.5]{Eichelsbacher2010}. Let $(I,J,K,L)$ be distributed as in
the proof of Theorem~\ref{thm3}, independent of all else. Using the notation
from~\ref{sec3} and the proof of Theorem~\ref{thm3} (with respect to
exchangeability of $X$), define the vectors
\be
	X' = X^{(I)},\qquad
	X'' = X^{IJKL}.
\ee
Set $\~D=D$ and define $G$ as
\be
	G_k = -n^{3/2}\delta_{kK}\bbbklr{\frac{\beta}{n(1-
          \beta)}+\delta_{kI}}\bklr{\sigma_I-\beta m}.
\ee
Define the matrix $S$ as 
\be
	S_{kl} = n^2\delta_{kK}\delta_{lI}s_{IK}.
\ee
Define now $f_\cdot$ as in the proof of Theorem~\ref{thm3}. Then,
\bes
  & - \IE \sum_{k} G_kf_k(X) \\
  & \qquad= n^{3/2}\IE \bbbklr{\frac{\beta}{n(1- \beta)}
            +\delta_{KI}}\bklr{\sigma_I-\beta m}f_K(X)\\
  & \qquad = n^{-1/2}\IE \sum_{i,k}\bbbklr{\frac{\beta}{n(1- \beta)}
                  +\delta_{ki}}\bklr{\sigma_i-\beta m}f_k(X)\\
  & \qquad = n^{-1/2}\IE \sum_{i,k}\bbbklr{\frac{\beta\sigma_i}{n(1- \beta)}
                  -\frac{\beta^2m}{n(1- \beta)}
                  +\delta_{ki}\sigma_i-\delta_{ki}\beta m}f_k(X)\\
  & \qquad = n^{-1/2}\IE \sum_{k}\bbbklr{\frac{\beta m }{1-\beta}
                -\frac{\beta^2m}{1- \beta}+\sigma_k-\beta m}f_k(X)\\
  & \qquad = n^{-1/2}\IE \sum_{k}\sigma_k f_k(X) =\IE \sum_{k}X_k f_k(X),
\ee
and
\bes
  & \IE \sum_{k} G_kf_k(X')\\
  & \qquad=  -n^{-3/2} \IE \bbbklr{\frac{\beta}{n(1-
          \beta)}+\delta_{KI}}\bklr{\sigma_I-\beta m}f_K(X^{(I)})\\
  & \qquad = -n^{-1/2}\IE \sum_{i,k}\bbbklr{\frac{\beta}{n(1-
          \beta)}+\delta_{ki}}\bklr{\sigma_i-\beta m}f_k(X^{(i)})\\
  & \qquad = -n^{-1/2}\IE \sum_{i,k}\bbbklr{\frac{\beta}{n(1-
          \beta)}+\delta_{ki}}\bklr{\tanh(\beta m_i)-\beta m}f_k(X^{(i)}),
\ee
where for the last equation we used that $\IE(\sigma_i|X^{(i)})=\tanh(\beta
m_i)$.
Using the estimate
\bes
  \abs{\tanh(\beta m_i) - \beta m} &\leq
    \abs{\tanh(\beta m_i) - \tanh(\beta m)} + \abs{\tanh(\beta m) - \beta m}\\
    &\leq \frac{\beta}{n} + \frac{\beta^3 \abs{m}^3}{6}.
\ee
we obtain
\be
  \bbabs{\IE \sum_{k} G_kf_k(X')}\leq
    \abs{f}_1\frac{\bklr{6+\beta^2 n
        \IE\abs{m}^3}\bklr{\beta(1+\beta)}}{6(1-\beta)n^{1/2}}
\ee
and hence
\be
  \abs{\IE R_0(t)}\leq \abs{h}_1\frac{\bklr{6+\beta^2 n \IE m^3}
              \bklr{\beta(1+\beta)}}{6(1- \beta)n^{1/2}}
              \leq \frac{C_\beta\abs{h}_1}{n^{1/2}}.
\ee
(c.f.\ Remark~\ref{rem1}). Using exchangeability of $X$ for the second equation,
the fact that $\delta_{KI}=\delta_{JL}$ and also that $s_{IK} = s_{JL}$, we
obtain 
\bes
	& \IE \sum_{k,l} (G_k D_l-S_{kl}) f_{kl}(X'')\\
	& \quad = n^2\IE\bbbklg{\bbbklr{\frac{1}{n}\bbbklr{\frac{\beta}{n(1- \beta)}+\delta_{KI}}\bklr{\sigma_I-\beta m}\sigma_I- s_{IK}}f_{KI}(X^{IJKL})}\\
	& \quad = n^2\IE\bbbklg{\bbbklr{\frac{1}{n}\bbbklr{\frac{\beta}{n(1- \beta)}+\delta_{JL}}\bklr{\sigma_J-\beta m}\sigma_J- s_{JL}}f_{KI}(X)}\\
		& \quad = \frac{1}{n}\IE\sum_{i,j}\bbbklg{\bbbklr{\frac{1}{n}\bbbklr{\frac{\beta}{n(1- \beta)}+1}\bklr{\sigma_j-\beta m}\sigma_j- s_{jj}}f_{ii}(X)}\\
		& \qquad + \frac{1}{n(n-1)}\IE\sum_{i,j,k\neq i,l\neq j}\bbbklg{\bbbklr{\frac{1}{n}\bbbklr{\frac{\beta}{n(1- \beta)}}\bklr{\sigma_j-\beta m}\sigma_j -s_{jl}}f_{ki}(X)}\\
		& \quad = \IE\bbbklg{\sum_{j}\bbbklr{\frac{1}{n}\bbbklr{\frac{\beta}{n(1- \beta)}+1}\bklr{1-\beta \sigma_j m}- s_{jj}}\sum_{i}\frac{f_{ii}(X)}{n}}\\
		& \qquad + \IE\bbbklg{\sum_{j,l\neq j}\bbbklr{\frac{1}{n}\bbbklr{\frac{\beta}{n(1- \beta)}}\bklr{1-\beta \sigma_j m} -s_{jl}}\sum_{i,k\neq i}\frac{f_{ki}(X)}{n(n-1)}}\\
		& \quad = -\IE\bbbklg{\frac{\bklr{\beta+n(1-\beta)}\beta m^2}{n(1- \beta)}\sum_{i}\frac{f_{ii}(X)}{n}
		- \frac{(n-1)\beta^2 m^2}{n(1- \beta)}\sum_{i,k\neq i}\frac{f_{ki}(X)}{n(n-1)}}.
\ee
Thus,
\be
	\abs{\IE R_1(t)}\leq C_\beta\abs{h}_2\IE m^2
	\leq \frac{C_\beta \abs{h}_2}{n}.
\ee
Now
\bes
	\IE R_2(t,u) & = \IE\sum_{k,l,m} (G_kD_l-S_{kl})D'_m f_{klm}(X+uD')\\
		&=n^{3/2}\IE\sum_{m\in\{I,J,K,L\}}\bbbkle{\frac{1}{n}
				\bbbklr{\frac{\beta}{n(1- \beta)}+\delta_{KI}}\bklr{\sigma_I-\beta m}
							\sigma_I- s_{IK}}\sigma_m \\
		&\kern20em \times f_{KIm}(X+uD').
\ee
From this, it is not difficult to see that 
\be
	\abs{\IE R_2(t,u)}\leq \frac{C_\beta \abs{h}_3}{n^{1/2}}
\ee
(recall the definition of $s_{ij}$ and note that the probability that $I=K$ is $1/n$).
Similarly,
\be
	\abs{\IE R_3(t,u)}\leq \frac{C_\beta \abs{h}_3}{n^{1/2}}.
\ee
Putting all the estimates together, yields the claim.
\end{proof}

\section*{Acknowledgements}

The author thanks the anonymous referees for helpful comments and literature
references.


\begin{thebibliography}{30}
\providecommand{\natexlab}[1]{#1}
\providecommand{\url}[1]{\texttt{#1}}
\expandafter\ifx\csname urlstyle\endcsname\relax
\providecommand{\doi}[1]{doi: #1}\else
\providecommand{\doi}{doi: \begingroup \urlstyle{rm}\Url}\fi

\bibitem[Baik and Suidan(2005)]{Baik2005}
J.~Baik and T.~M. Suidan (2005).
\newblock A {GUE} central limit theorem and universality of directed first and
last passage site percolation.
\newblock \emph{Int. Math. Res. Not.}pages 325--337.

\bibitem[Barbour et~al.(1989)Barbour, Karo{\'n}ski, and
Ruci{\'n}ski]{Barbour1989}
A.~D. Barbour, M.~Karo{\'n}ski, and A.~Ruci{\'n}ski (1989).
\newblock A central limit theorem for decomposable random variables with
applications to random graphs.
\newblock \emph{J.~Combin. Theory Ser.~B} \textbf{47}, \penalty0 125--145.

\bibitem[Bodineau and Martin(2005)]{Bodineau2005}
T.~Bodineau and J.~Martin (2005).
\newblock A universality property for last-passage percolation paths close to
the axis.
\newblock \emph{Electron. Comm. Probab.} \textbf{10}, \penalty0 105--112
(electronic).

\bibitem[Bolthausen(1982)]{Bolthausen1982a}
E.~Bolthausen (1982).
\newblock Exact convergence rates in some martingale central limit theorems.
\newblock \emph{Ann. Probab.} \textbf{10}, \penalty0 672--688.

\bibitem[Carmona and Hu(2006)]{Carmona2006}
P.~Carmona and Y.~Hu (2006).
\newblock Universality in {S}herrington-{K}irkpatrick's spin glass model.
\newblock \emph{Ann. Inst. H. Poincar{\'e} Probab. Statist.} \textbf{42},
\penalty0 215--222.

\bibitem[Chatterjee(2005)]{Chatterjee2005a}
S.~Chatterjee (2005).
\newblock A simple invariance theorem.
\newblock Available at \url{http://arxiv.org/abs/math.PR/0508213}.

\bibitem[Chatterjee(2007)]{Chatterjee2007a}
S.~Chatterjee (2007).
\newblock A generalization of the {L}indeberg principle.
\newblock \emph{Ann. Probab.} \textbf{34}, \penalty0 2061--2076.

\bibitem[Chatterjee and Meckes(2008)]{Chatterjee2008a}
S.~Chatterjee and E.~Meckes (2008).
\newblock Multivariate normal approximation using exchangeable pairs.
\newblock \emph{ALEA Lat. Am. J. Probab. Math. Stat.} \textbf{4}, \penalty0
257--283.
\newblock ISSN 1980-0436.

\bibitem[Chatterjee and Shao(to appear)]{Chatterjee2010}
S.~Chatterjee and Q.-M. Shao (to appear).
\newblock Non-normal approximation by {S}tein's method of exchangeable pairs
with application to the {C}urie-{W}eiss model.
\newblock \emph{Ann. App. Probab.}

\bibitem[Chen and R\"ollin(2010)]{Chen2010b}
L.~H.~Y. Chen and A.~R\"ollin (2010).
\newblock Stein couplings for normal approximation.
\newblock \emph{Preprint}.
\newblock Available at \url{http://arxiv.org/abs/1003.6039}.

\bibitem[Dembo and Rinott(1996)]{Dembo1996}
A.~Dembo and Y.~Rinott (1996).
\newblock Some examples of normal approximations by {S}tein's method.
\newblock In \emph{Random discrete structures (Minneapolis, MN, 1993)},
volume~76 of \emph{IMA Vol. Math. Appl.}, pages 25--44. Springer, New York.

\bibitem[Eichelsbacher and L{\"o}we(2010)]{Eichelsbacher2010}
P.~Eichelsbacher and M.~L{\"o}we (2010).
\newblock Stein's method for dependent random variables occurring in
statistical mechanics.
\newblock \emph{Electron. J. Probab.} \textbf{15}, \penalty0 962--988.

\bibitem[Erd{\H{o}}s et~al.(2010)Erd{\H{o}}s, Yau, and Yin]{Erdos2010}
L.~Erd{\H{o}}s, H.-T. Yau, and J.~Yin (2010).
\newblock Bulk universality for generalized {W}igner matrices.
\newblock \emph{Preprint}.

\bibitem[G{\"o}tze and Tikhomirov(2006)]{Gotze2006}
F.~G{\"o}tze and A.~N. Tikhomirov (2006).
\newblock Limit theorems for spectra of random matrices with martingale
structure.
\newblock \emph{Teor. Veroyatn. Primen.} \textbf{51}, \penalty0 171--192.

\bibitem[Grama(1997)]{Grama1997}
I.~G. Grama (1997).
\newblock On moderate deviations for martingales.
\newblock \emph{Ann. Probab.} \textbf{25}, \penalty0 152--183.
\newblock ISSN 0091-1798.

\bibitem[Johansson(2000)]{Johansson2000}
K.~Johansson (2000).
\newblock Shape fluctuations and random matrices.
\newblock \emph{Comm. Math. Phys.} \textbf{209}, \penalty0 437--476.

\bibitem[Johansson(2001)]{Johansson2001}
K.~Johansson (2001).
\newblock Universality of the local spacing distribution in certain ensembles
of {H}ermitian {W}igner matrices.
\newblock \emph{Comm. Math. Phys.} \textbf{215}, \penalty0 683--705.
\newblock ISSN 0010-3616.
\newblock \doi{10.1007/s002200000328}.
\newblock Available at
\url{http://dx.doi.org.libproxy1.nus.edu.sg/10.1007/s002200000328}.

\bibitem[Meckes(2009)]{Meckes2009}
E.~S. Meckes (2009).
\newblock On {S}tein's method for multivariate normal approximation.
\newblock \emph{to appear in High Dimensional Probability V}.

\bibitem[Mossel et~al.(2010)Mossel, O'Donnell, and Oleszkiewicz]{Mossel2010}
E.~Mossel, R.~O'Donnell, and K.~Oleszkiewicz (2010).
\newblock Noise stability of functions with low influences: invariance and
optimality.
\newblock \emph{Ann. of Math.} \textbf{171}, \penalty0 295--341.

\bibitem[Rai{\v{c}}(2004)]{Raic2004}
M.~Rai{\v{c}} (2004).
\newblock A multivariate {CLT} for decomposable random vectors with finite
second moments.
\newblock \emph{J. Theoret. Probab.} \textbf{17}, \penalty0 573--603.

\bibitem[Reinert and R{\"o}llin(2009)]{Reinert2009}
G.~Reinert and A.~R{\"o}llin (2009).
\newblock Multivariate normal approximation with {S}tein's method of
exchangeable pairs under a general linearity condition.
\newblock \emph{Ann. Probab.} \textbf{37}, \penalty0 2150--2173.

\bibitem[Rinott and Rotar{$'$}(1999)]{Rinott1999}
I.~Rinott and V.~I. Rotar{$'$} (1999).
\newblock Some estimates for the rate of convergence in the {CLT} for
martingales. {II}.
\newblock \emph{Theory Probab. Appl.} \textbf{44}, \penalty0 523--536.

\bibitem[Rinott and Rotar{$'$}(1996)]{Rinott1996}
Y.~Rinott and V.~Rotar{$'$} (1996).
\newblock A multivariate {CLT} for local dependence with {$n\sp {-1/2}\log n$}
rate and applications to multivariate graph related statistics.
\newblock \emph{J. Multivariate Anal.} \textbf{56}, \penalty0 333--350.

\bibitem[Rotar{$'$}(1973)]{Rotar1973}
V.~I. Rotar{$'$} (1973).
\newblock Certain limit theorems for polynomials of degree two.
\newblock \emph{Teor. Verojatnost. i Primenen.} \textbf{18}, \penalty0
527--534.
\newblock Actual title is "Some limit theorems for polynomials of degree two".

\bibitem[Slepian(1962)]{Slepian1962}
D.~Slepian (1962).
\newblock The one-sided barrier problem for {G}aussian noise.
\newblock \emph{Bell System Tech. J.} \textbf{41}, \penalty0 463--501.
\newblock ISSN 0005-8580.

\bibitem[Stein(1972)]{Stein1972}
C.~Stein (1972).
\newblock A bound for the error in the normal approximation to the distribution
of a sum of dependent random variables.
\newblock In \emph{Proceedings of the Sixth Berkeley Symposium on Mathematical
Statistics and Probability (Univ. California, Berkeley, Calif., 1970/1971),
Vol. II: Probability theory}, pages 583--602, Berkeley, Calif. Univ.
California Press.

\bibitem[Suidan(2006)]{Suidan2006}
T.~Suidan (2006).
\newblock A remark on a theorem of chatterjee and last passage percolation.
\newblock \emph{J. Phys. A: Math. Gen.} \textbf{39}, \penalty0 8977--8981.

\bibitem[Talagrand(2006)]{Talagrand2006}
M.~Talagrand (2006).
\newblock The {P}arisi formula.
\newblock \emph{Ann. of Math.} \textbf{163}, \penalty0 221--263.

\bibitem[Talagrand(2010)]{Talagrand2010}
M.~Talagrand (2010).
\newblock \emph{Mean field models for spin glasses}.
\newblock Springer-Verlag.

\bibitem[Tao and Vu(2010)]{Tao2010b}
T.~Tao and V.~Vu (2010).
\newblock Random matrices: universality of local eigenvalue statistics.
\newblock \emph{Acta Math.}

\end{thebibliography}

\end{document}